\documentclass[final,12pt,a4paper]{amsart} 

\usepackage{graphicx}
\usepackage[all]{xy}
\usepackage{placeins}
\usepackage{enumitem}
\usepackage{amssymb}
\usepackage{latexsym}
\usepackage{amsmath}
\usepackage{mathrsfs}
\usepackage{array,booktabs}
\usepackage{verbatim}
\usepackage{fullpage}
\usepackage[normalem]{ulem}

\usepackage[usenames,dvipsnames]{xcolor}
\usepackage{hyperref}
\hypersetup{
    colorlinks,
    linkcolor={red!50!black},
    citecolor={blue!50!black},
    urlcolor={blue!80!black}
}

\newcommand{\harxiv}[1]{ \href{http://arxiv.org/abs/#1}{\texttt{arXiv:#1}}}
\newcommand{\hyref}[2]{ \hyperref[#2]{#1~\ref*{#2}} }

\usepackage[notref, notcite]{showkeys} 

\usepackage{tikz}
\usetikzlibrary{decorations.pathmorphing}
\usetikzlibrary{decorations.pathreplacing}
\usetikzlibrary{positioning}
\usetikzlibrary{calc}
\usetikzlibrary{backgrounds}
\usetikzlibrary{shapes.misc}
\usetikzlibrary{patterns}

\usepackage{color}

\theoremstyle{plain}
\newtheorem{theorem}{Theorem}[section]

\newtheorem{lemma}[theorem]{Lemma}
\newtheorem{corollary}[theorem]{Corollary}
\newtheorem{proposition}[theorem]{Proposition}
\newtheorem{introtheorem}{Theorem}

\theoremstyle{definition}
\newtheorem{remark}[theorem]{Remark}
\newtheorem{example}[theorem]{Example}
\newtheorem{definition}[theorem]{Definition}
\newtheorem{setup}[theorem]{Setup}

\newcommand{\sA}{\mathsf{A}}
\newcommand{\sB}{\mathsf{B}}
\newcommand{\sC}{\mathsf{C}}
\newcommand{\sD}{\mathsf{D}}

\newcommand{\sH}{\mathsf{H}}
\newcommand{\sK}{\mathsf{K}}

\newcommand{\sS}{\mathsf{S}}
\newcommand{\sT}{\mathsf{T}}
\newcommand{\sU}{\mathsf{U}}
\newcommand{\sV}{\mathsf{V}}

\DeclareMathAlphabet{\mathpzc}{OT1}{pzc}{m}{it}

\newcommand{\bE}{\mathbb{E}}

\newcommand{\fs}{\mathfrak{s}}

\newcommand{\cF}{\mathcal{F}}
\newcommand{\cT}{\mathcal{T}}
\newcommand{\cW}{\mathcal{W}}
\newcommand{\cX}{\mathcal{X}}
\newcommand{\cY}{\mathcal{Y}}

\newcommand{\Db}{\sD^b}
\newcommand{\Kb}{\sK^b}
\newcommand{\Ab}{\mathsf{Ab}}
\newcommand{\proj}{\mathsf{proj}}
\renewcommand{\mod}{\mathsf{mod}\,}
\DeclareMathOperator{\Mod}{\mathsf{Mod}}

\DeclareMathOperator{\im}{\mathrm{im}}

\newcommand{\kk}{\mathbf{k}}

\DeclareMathOperator{\Hom}{\mathrm{Hom}}
\DeclareMathOperator{\End}{\mathrm{End}}

\DeclareMathOperator{\add}{\mathsf{add}}
\DeclareMathOperator{\thick}{\mathsf{thick}}
\newcommand{\orth}{{}^\perp}
\newcommand{\op}{{}^{\rm op}}

\newcommand{\too}{\longrightarrow}

\newcommand{\rightlabel}[1]{\stackrel{#1}{\too}}
\newcommand{\rightiso}{\stackrel{\sim}{\too}}
\newcommand{\longmapsfrom}{\longleftarrow\!\shortmid}
\renewcommand{\longmapsto}{\shortmid\!\longrightarrow}

\newcommand{\coloneqq}{:=}

\newcommand{\bij}{\stackrel{1-1}{\longleftrightarrow}}

\renewcommand{\phi}{\varphi}
\renewcommand{\epsilon}{\varepsilon}

\begin{document}

\title{Co-t-structures, cotilting and cotorsion pairs}

\author{David Pauksztello} 
\address{Department of Mathematics and Statistics, Lancaster University, Lancaster, LA1 4YF, United Kingdom.}
\email{d.pauksztello@lancaster.ac.uk}

\author{Alexandra Zvonareva}
\address{Institut f\"ur Algebra und Zahlentheorie, Universit\"at Stuttgart, Pfaffenwaldring 57, 70569 Stuttgart, Germany.}
\email{alexandrazvonareva@gmail.com}



\begin{abstract}
Let $\sT$ be a
triangulated category with shift functor $\Sigma \colon \sT \to \sT$.
Suppose $(\sA,\sB)$ is a co-t-structure with coheart $\sS = \Sigma \sA \cap \sB$ and extended coheart $\sC = \Sigma^2 \sA \cap \sB = \sS* \Sigma \sS$, which is an extriangulated category. We show that there is a bijection between co-t-structures $(\sA',\sB')$ in $\sT$ such that $\sA \subseteq \sA' \subseteq \Sigma \sA$ and complete cotorsion pairs in the extended coheart $\sC$.
In the case that $\sT$ is Hom-finite, $\kk$-linear and Krull-Schmidt, we show further that there is a bijection between complete cotorsion pairs in $\sC$ and functorially finite torsion pairs in $\mod \sS$.
\end{abstract}

\maketitle


\section*{Introduction} 

Happel-Reiten-Smal\o\ (HRS) tilting was introduced in \cite{HRS} as a method to construct new t-structures from torsion pairs in the heart of a given t-structure. 
Suppose $\sT$ is a triangulated category with shift functor $\Sigma \colon \sT \to \sT$,
and $(\sU,\sV)$ is a t-structure in $\sT$ with the heart $\sH = \sU \cap \Sigma \sV$.
Given a torsion pair $(\cT,\cF)$ in $\sH$ the HRS tilt of $(\sU,\sV)$ at $(\cT,\cF)$ is the t-structure
\[
(\sU',\sV') \coloneqq  (\Sigma \sU * \cT, \cF * \sV).
\]
%
In addition to providing a method for constructing new t-structures from old, HRS tilting gives all t-structures that are `sufficiently close' to the initial one; see \cite{BR,Pol,Woolf}. Explicitly, there is  a bijection:
\begin{equation} \label{hrs}
\{
\text{t-structures } (\sU',\sV') \text{ with } \Sigma \sU \subseteq \sU' \subseteq \sU
\}
\bij
\{
\text{torsion pairs } (\cT,\cF) \text{ in } \sH
\}.
\end{equation}

Such t-structures $ (\sU',\sV')$ are often called \emph{intermediate with respect to $ (\sU,\sV)$}. 

HRS tilting has many applications in representation theory and algebraic geometry.
For example, it provides a method for constructing derived equivalences between abelian categories in cases where explicit tilting objects are not available. In this context, HRS tilting was used to study derived equivalences for smooth compact analytic surfaces with no curves \cite{BVdB} or for K3 surfaces  \cite{Br,Hub}. 
Recently, HRS tilting has been extensively used in the study of Bridgeland stability conditions \cite{Br,PSZ,QW,Woolf}.

A \emph{co-t-structure} in $\sT$ consists of a pair of full subcategories $(\sA,\sB)$ of $\sT$ which are closed under direct summands, such that $\Sigma^{-1} \sA \subseteq \sA$, $\sT(a,b) = 0$, and $\sT = \sA * \sB$ \cite{Bondarko,Pauk}; note that in \cite{Bondarko} they are called weight structures. The subcategory $\sS = \Sigma \sA \cap \sB$ is called the \emph{coheart}; it is a \emph{presilting subcategory} of $\sT$, see Section~\ref{sec:background}.
Since their introduction, co-t-structures have acquired an important role in representation theory in connection with silting theory and $\tau$-tilting theory \cite{AIR,AI,IJY,KV,KY,MSSS}; for surveys of recent results see \cite{Angeleri,Jorg-co-t}.
At first sight, the definitions of t-structure and co-t-structure appear very similar and there are, indeed, a number of parallels between the two theories. However, t-structures and co-t-structures are not dual to each other in a mathematical sense and there are notable differences between them, with the most basic being the failure, of abelianness of the coheart. 

The main result of this note is an analogue of bijection \eqref{hrs} for co-t-structures. A priori it is not clear what are the co-t-structure counterparts of the HRS tilting procedure and the torsion pair in the heart. 
The recent introduction of extriangulated categories in \cite{NP} provides the right context.
If $(\sA,\sB)$ is a co-t-structure in $\sT$ with coheart $\sS$, then $\sC = \Sigma^2\sA \cap \sB =\sS * \Sigma \sS$, which we call the \emph{extended coheart}, is an extriangulated category, in which there is a notion of a \emph{complete cotorsion pair} \cite{Hovey,NP,Salce}.
A particular example of the extended coheart appears in the context of Amiot cluster categories in the guise of the fundamental domain of the cluster category \cite{Amiot}. 

\begin{introtheorem}[Theorem~\ref{thm:co-hrs}] \label{intro:co-hrs}
Suppose $\sT$ is a
triangulated category,
$(\sA,\sB)$ is a co-t-structure in $\sT$, and $\sC$ is its extended coheart. There is a bijection
\[
\{\text{co-t-structures } (\sA',\sB') \text{ with } \sA \subseteq \sA' \subseteq \Sigma \sA\}
\bij
\{\text{complete cotorsion pairs } (\cX,\cY) \text{ in } \sC\}.
\]
\end{introtheorem}

Let $A$ be a finite-dimensional $\kk$-algebra. From \cite{KY}, there are bijections between the following objects:
\begin{itemize}
\item silting subcategories, $\sS$, in $\Kb(\proj(A))$; 
\item bounded co-t-structures, $(\sA_\sS,\sB_\sS)$, in $\Kb(\proj(A))$;
\item \emph{algebraic t-structures}, i.e. bounded t-structures, $(\sU_\sS,\sV_\sS)$, in $\Db(A)$ with length heart.
\end{itemize}
Fixing a silting subcategory $\sS = \add(s)$ for a silting object $s$, by \cite{IJY} these bijections restrict to bijections between 
\begin{itemize}
\item algebraic t-structures intermediate with respect to $(\sU_\sS,\sV_\sS)$;
\item bounded co-t-structures intermediate with respect to $(\sA_\sS,\sB_\sS)$; and,
\item silting subcategories $\sS'$ with $\sS' \subseteq \sS * \Sigma \sS$.
\end{itemize}
Finally, HRS tilting and support $\tau$-tilting theory \cite{AIR,IJY} adds a bijection with
\begin{itemize}
\item functorially finite torsion pairs in $\mod \End(s)\op \simeq \mod \sS$
\end{itemize}
into the mix. 
Theorem \ref{intro:co-hrs} completes the picture with the co-t-structure version of torsion pairs: cotorsion pairs in $\sS*\Sigma \sS$.
Moreover, working with co-t-structures and cotorsion pairs seems to provide a more convenient context for representation theory: one does not have to care about the additional restriction on the t-structure having a length heart, which may be difficult to check in practice, cf. \cite{CSPP}.

Our second result provides a direct and explicit connection between cotorsion pairs in $\sS * \Sigma \sS$ and torsion pairs in $ \mod \sS$. 

\begin{introtheorem}[Theorem~\ref{thm:cotorsion-torsion}] \label{intro:cotorsion-torsion}
Suppose $\sT$ is an essentially small, Hom-finite, $\kk$-linear, Krull-Schmidt triangulated category.
If $\sS = \add(s)$ is a presilting subcategory of $\sT$ and $\sC = \sS * \Sigma \sS$, then the restricted Yoneda functor, $F \colon \sC \to \mod \sS$, induces a bijection
\[
\{\text{complete cotorsion pairs in } \sC \} \bij
\{\text{functorially finite torsion pairs in } \mod \sS \}.
\]
In particular, $F$ sends cotorsionfree classes to torsion classes.
\end{introtheorem}

We note that the connections between torsion and cotorsion pairs have been studied before in a different setting \cite{BR}. In particular, Beligiannis and Reiten consider cotorsion pairs in abelian categories and the corresponding cotorsion pairs in pretriangulated categories.

\subsection*{Acknowledgments}
The authors would like to thank Steffen K\"onig for reading a preliminary version of this article.
The authors gratefully acknowledge support from the Representation Theory Group at Universit\"at Stuttgart and the Lancaster University Department of Mathematics and Statistics' Visitor Fund to enable parts of this work to be carried out. 
Both authors are grateful for the hospitality they received in Lancaster and Stuttgart.

\section{Background} \label{sec:background}

Let $\sA$ be an additive category and $\sB \subset \sA$ a subcategory. For objects $a_1$, $a_2$ of $A$ we will write $\sA(a_1,a_2) = \Hom_\sA(a_1,a_2)$. We define the left and right orthogonal categories of $\sB$ as follows:
\[
{}\orth \sB := \{a \in \sA \mid \sA(a,b) = 0 \text{ for all } b \in \sB\}
\text{ and }
\sB\orth := \{a \in \sA \mid \sA(b,a) = 0 \text{ for all } b \in \sB\}.
\]
We will often use the shorthand $\sA(a,\sB) = 0$ to mean $\sA(a,b) = 0$ for all $b \in \sB$; similarly for the shorthand $\sA(\sB,a)$.

Throughout this note $\sT$ will be a triangulated category with shift functor $\Sigma \colon \sT \to \sT$. 
 For two subcategories $\sA$, $\sB$ of $\sT$ the full subcategory with objects $\{t\mid $ there exists a triangle $a \to t \to b \to \Sigma a$ with $a \in \sA$ and $b \in \sB\}$ will be denoted by $\sA * \sB$. A full additive subcategory $\sC$ of $\sT$ is \emph{extension-closed} if $\sC * \sC = \sC$.

\subsection{Approximations}

Let $\sA$ be a subcategory of $\sT$ and let $t$ be an object of $\sT$. A morphism $f \colon t \to a$ with $a \in \sA$ is called 
\begin{itemize}
\item a \emph{left $\sA$-approximation of $t$} if $\sT(f,\sA) \colon \sT(a,\sA) \to \sT(t,\sA)$ is surjective;
\item \emph{left minimal} if any $g \colon a \to a$ such that $gf = f$ is an automorphism; and,
\item a \emph{minimal left $\sA$-approximation of $t$} if it is both left minimal and a left $\sA$-approximation of $t$. 
\end{itemize}
Left $\sA$-approximations are sometimes called \emph{$\sA$-pre-envelopes}.
If every object of $\sT$ admits a left $\sA$-approximation then $\sA$ is said to be \emph{covariantly finite} in $\sT$.
There is a dual notion of a \emph{(minimal) right $\sA$-approximation} (or an \emph{$\sA$-precover}); if every object of $\sT$ admits a right $\sA$-approximation, then $\sA$ is said to be \emph{contravariantly finite} in $\sT$.

Minimal approximations admit the following important property; see, for example, \cite{Jorgensen} for a triangulated version. We give the statement for left approximations; there is a dual statement for right approximations.

\begin{lemma}[Wakamatsu lemma for triangulated categories] \label{lem:wakamatsu}
Let $\sA$ be an extension closed subcategory of $\sT$ and suppose $f \colon t \to a$ is a minimal left $\sA$-approximation of $t$. Then in the triangle
\[
b \too t \rightlabel{f} a \to \Sigma b,
\]
we have $b \in {}\orth\sA$.
\end{lemma}

\subsection{Co-t-structures, silting subcategories and the extended coheart}

We recall the following definitions from \cite{AI,KV} and \cite{Bondarko}, respectively.

\begin{definition}
A subcategory $\sS$ of a triangulated category $\sT$ is \emph{presilting} if $\sT(\sS, \Sigma^i \sS) = 0$ for all $i > 0$; it is called \emph{silting} if, in addition $\thick \sS = \sT$, where $\thick \sS$ is the smallest triangulated subcategory of $\sT$ containing $\sS$ that is closed under direct summands.
An object $s$ of $\sT$ is a \emph{(pre)silting object} if $\add(s)$ is a (pre)silting subcategory, where $\add(s)$ consists of the direct summands of finite direct sums of copies of $s$.
\end{definition}

\begin{definition}
A co-t-structure $(\sA,\sB)$ in $\sT$ is \emph{bounded} if $\bigcup_{i\in \mathbb{Z}}\Sigma^i\sA=\sT=\bigcup_{i\in \mathbb{Z}}\Sigma^i\sB$.
\end{definition}

The coheart $\sS = \Sigma \sA \cap \sB$ of a co-t-structure $(\sA,\sB)$ is always a presilting subcategory. It is silting precisely when the co-t-structure is bounded \cite[Corollary 5.9]{MSSS}.

\begin{definition}
Let $(\sA,\sB)$ be a co-t-structure in $\sT$. The subcategory $\sC = \Sigma^2 \sA \cap \sB$ will be called the \emph{extended coheart} of the co-t-structure.
\end{definition}

The following lemma shows that the extended coheart of $(\sA,\sB)$ consists of precisely the objects of $\sT$ which are `two-term' with respect to the coheart $\sS$.

\begin{lemma}[{\cite[Lemma 2.1]{IJY}}]
Let $(\sA,\sB)$ be a co-t-structure in $\sT$ with coheart $\sS$. Then the extended coheart $\sC = \Sigma^2 \sA \cap \sB = \sS * \Sigma \sS$.
\end{lemma}

\subsection{Extriangulated categories and complete cotorsion pairs}

We will use the notion of an extriangulated category from \cite{NP} without recalling the complete definition. 

An \emph{extriangulated category} consists of a triple $(\sC, \bE, \fs)$, where $\sC$ is an additive category,  $\bE(-,-) \colon \sC^{op} \times \sC \to \Ab$ is a biadditive functor and $\fs$ assigns to any element of $\bE(c,a)$ an equivalence class of pairs of morphisms $[a\to b\to c]$, called an \emph{$\bE$-triangle}. In addition the triple $(\sC, \bE, \fs)$ should satisfy a number of axioms reminiscent of the axioms of a triangulated category (without rotation of triangles). 

If $\sC$ is an additive category, $\Sigma$ is an equivalence on $\sC$ and $\bE:=\sC(-,\Sigma -)$, then by \cite[Proposition 3.22]{NP} fixing a triangulated structure on $\sC$ with the shift functor $\Sigma$ is equivalent to fixing an extriangulated structure on $\sC$ with the additive bifunctor $\bE$, where $\fs$ assigns to an element $\delta\in\sC(c,\Sigma a)$ the isomorphism class of distinguished triangles $a\to b \to c \xrightarrow{\delta} \Sigma a$. All extriangulated categories used in this paper will be subcategories of triangulated categories with the induced extriangulated structure, that is $\bE(-,-)$ is the restriction of $\sC(-,\Sigma -)$ and $\bE$-triangles are distinguished triangles $a\to b \to c \xrightarrow{\delta} \Sigma a$ with $a,b,c$ in the subcategory. Analogously to triangulated and exact categories, a subcategory $\sC$ of an extriangulated category is called extension-closed, if for any $\bE$-triangle $a\to b \to c \xrightarrow{\delta} \Sigma a$ with $a,c\in \sC$ the object $b$ is also in $\sC$.

\begin{lemma}
Let $(\sA,\sB)$ be a co-t-structure in $\sT$. Then $\sC = \Sigma^2 \sA \cap \sB$ is an extriangulated category with the extriangulated structure induced by the triangulated structure of $\sT$.
\end{lemma}

\begin{proof}
Since both $\sA$ and $\sB$ are extension-closed subcategories of $\sT$, we have that $\sC$ is an extension-closed subcategory of $\sT$.  The triangulated structure on $\sT$ also provides an extriangulated structure on $\sT$, see \cite[Example 2.13]{NP}. Hence, by \cite[Remark 2.18]{NP}, the triangulated structure on $\sT$ restricted to $\sC$ induces an extriangulated structure on $\sC$.
\end{proof}

We transpose the following definitions from the exact and abelian settings (see \cite{Hovey} and \cite{Salce}) to the extriangulated setting.

\begin{definition}
Let $(\sC,\bE,\fs)$ be an extriangulated category. A \emph{cotorsion pair} in $\sC$ consists of a pair of full additive subcategories $(\cX,\cY)$ closed under direct summands such that for each $c \in \sC$ the following holds:
\begin{enumerate}[label=(\arabic*)]
\item $c \in \cX$ if and only if $\bE(c,\cY) = 0$; and,
\item $c \in \cY$ if and only if $\bE(\cX,c) = 0$.
\end{enumerate}
\end{definition}

Since $\cX$ and $\cY$ are each realised as orthogonal subcategories they are closed under extensions. Indeed, by \cite[Proposition 3.3]{NP} any $\bE$-triangle $a\to b\to c$ gives rise to an exact sequence $\sC(-,a) \to \sC(-,b) \to \sC(-,c) \to \bE(-,a) \to \bE(-,b) \to \bE(-,c)$ and its dual.

\begin{definition}[{\cite[Definition 4.1]{NP}}] \label{DefCotorsion}
Let $(\sC,\bE,\fs)$ be an extriangulated category. A \emph{complete cotorsion pair} in $\sC$ consists of a pair of full additive subcategories $(\cX,\cY)$ closed under direct summands such that the following hold:
\begin{enumerate}[label=(\arabic*)]
\item \label{ext-orth} for each $x \in \cX$ and $y \in \cY$ we have $\bE(x,y) = 0$;
\item \label{E-triangle-CY} for each $c \in \sC$ there is an $\bE$-triangle $c \to y \to x$ with $x \in \cX$ and $y \in \cY$; and,
\item \label{E-triangle-XC} for each $c \in \sC$ there is an $\bE$-triangle $y \to x \to c$ with $x \in \cX$ and $y \in \cY$.
\end{enumerate}
A pair of full subcategories $(\cX,\cY)$ satisfying only condition \ref{ext-orth} will be called an \emph{Ext-orthogonal pair}.
\end{definition}

For each object $c$ of $\sC$, the morphism $c \to y$ occurring in the $\bE$-triangle above is always a left $\cY$-approximation of $c$. 
Similarly, the morphism $x \to c$ in the $\bE$-triangle above is a right $\cX$-approximation of $c$.

\begin{remark} \label{rem:cotorsion}
In this article we revert to the classical distinction between complete cotorsion pair and cotorsion pair in \cite{Hovey,Salce}. Therefore what is called a cotorsion pair in \cite{NP} will be called a complete cotorsion pair here. By \cite[Remark 4.4]{NP}, any complete cotorsion pair is a cotorsion pair, since the $0$ element of $\bE(c,a)$ is represented, up to equivalence, by a split $\bE$-triangle $a\to a\oplus c\to c$.
\end{remark}

\subsection{The restricted Yoneda functor} \label{sec:yoneda}

Assume now  that $\sT$ is essentially small, idempotent complete, Hom-finite, $\kk$-linear and Krull-Schmidt, where $\kk$ is a commutative noetherian ring. In this situation, Hom-finite means that $\sT(a,b)$ is a finitely-generated $\kk$-module for any $a,b\in \sT$. In particular, the endomorphism ring of an object $\sT(s,s)$ is a noetherian ring.
Suppose $\sS = \add(s)$ is a presilting subcategory of $\sT$ and let $\sC := \sS * \Sigma \sS$. 
We write $\Mod \sS$ for the category of contravariant additive functors from $\sS$ to the category $\Mod \kk$ and $\mod \sS$ for the full subcategory of finitely presented functors; see \cite{Auslander}.
Consider the restricted Yoneda functor
\begin{align*}
F \colon \sT & \to \Mod \sS. \\
t & \mapsto \sT(-,t)|_\sS
\end{align*}
By \cite[Proposition 6.2]{IY}, \cite[Remark 3.1]{IJY} the restricted Yoneda functor induces an equivalence of categories,
\[
F \colon (\sS * \Sigma \sS)/\Sigma \sS \to \mod \sS.
\]
Note there is an equivalence $\mod \sS \simeq \mod E$, where $E = \sT(s,s)$; see \cite[Remark 4.1]{IJY}. 

\subsection{Torsion pairs}

A \emph{torsion pair} on an abelian category $\sH$ consist of a pair of full subcategories $(\cT,\cF)$ of $\sH$ such that $\cT\orth = \cF$, ${}\orth \cF = \cT$ and for each object $h$ of  $\sH$ there is a short exact sequence 
\begin{equation} \label{tp}
0 \to t \to h \to f \to 0
\end{equation}
with $t \in \cT$ and $f \in \cF$. 
The subcategory $\cT$ is called the \emph{torsion class} and the subcategory $\cF$ is called the \emph{torsionfree class}.

By virtue of the short exact sequence \eqref{tp}, it follows that $\cT$ is contravariantly finite in $\sH$ and $\cF$ is covariantly finite in $\sH$.
If, in addition, $\cT$ is covariantly finite in $\sH$ (or, equivalently, $\cF$ is contravariantly finite in $\sH$ \cite[Theorem]{Smalo}), then we say that $(\cT,\cF)$ is a \emph{functorially finite torsion pair}; see e.g. \cite{AIR}.

If $\sH$ is noetherian, for example $\sH \simeq \mod E$ for a noetherian ring $E$, then any subcategory closed under extensions and quotients is a torsion class of a torsion pair; see, e.g. \cite[Chapter VI]{ASS} or \cite[Proposition 3.5]{Liu-Stanley}.
The dual statement holds for torsionfree classes.

\section{HRS tilting of co-t-structures at complete cotorsion pairs} \label{sec:co-hrs}

In this section $\sT$ will be an arbitrary triangulated category. 
The aim of this section is to prove Theorem~\ref{intro:co-hrs}.

\begin{theorem} \label{thm:co-hrs}
Suppose $\sT$ is a triangulated category,
$(\sA,\sB)$ is a co-t-structure in $\sT$, and $\sC = \Sigma^2 \sA \cap \sB$ is the extended coheart of $(\sA,\sB)$. 
Then there is a bijection
\begin{align*}
\{\text{co-t-structures } (\sA',\sB') \text{ with } \sA \subseteq \sA' \subseteq \Sigma \sA\}
& \bij
\{\text{complete cotorsion pairs } (\cX,\cY) \text{ in } \sC\}. \\
(\sA',\sB') 
& \longmapsto
(\sB\cap \Sigma \sA',\sB'\cap \Sigma^2 \sA) \\
(\add(\Sigma^{-1} \sA * \Sigma^{-1} \cX), \add(\cY * \Sigma^2 \sB))
& \longmapsfrom 
(\cX,\cY)
\end{align*}
\end{theorem}

\begin{remark}
Let $(\sA,\sB)$ be a co-t-structure in $\sT$. A co-t-structure $(\sA',\sB')$ such that $\sA \subseteq \sA' \subseteq \Sigma \sA$ (or, equivalently, $\sB \supseteq \sB' \supseteq \Sigma \sB'$) is often said to be \emph{intermediate with respect to $(\sA,\sB)$}; cf. \cite{Angeleri} or \cite{IJY}, in the former case `intermediate' means with respect to the `standard co-t-structure' and the interval may be larger.
\end{remark}

\begin{proof} 
The proof of this theorem consists of three steps: first we construct the map 
\[ 
\phi \colon \{\text{co-t-structures } (\sA',\sB') \text{ with } \sA \subseteq \sA' \subseteq \Sigma \sA\}
\to
\{\text{complete cotorsion pairs } (\cX,\cY) \text{ in } \sC\};
\]
then we construct the map $\psi$ in the opposite direction; then we prove that $\phi\psi=\mathrm{id}$ and $\psi\phi=\mathrm{id}$.

\textbf{Step 1:} Let  $(\sA',\sB')$ be  a co-t-structure in $\sT$ such that $\sA \subseteq \sA' \subseteq \Sigma \sA$ (equivalently $\Sigma \sB \subseteq \sB' \subseteq  \sB$) and consider the following subcategories of the extended coheart $\sC$:
\[
\cX:=\sB\cap \Sigma \sA'\subseteq \sC
\quad \text{and} \quad 
\cY:=\sB'\cap \Sigma^2 \sA \subseteq \sC.
\]
Note that $\cX$ and $\cY$ are closed under summands, since so are $\sA, \sA', \sB$ and $\sB'$. We claim that $(\cX,\cY)$ is a complete cotorsion pair in $\sC$. 
Since $\cX\in\Sigma \sA'$ and $\Sigma \cY\in \Sigma \sB'$, we have $\bE(\cX,\cY)=\sT(\cX,\Sigma \cY)=0$  and condition \ref{ext-orth} of Definition \ref{DefCotorsion} holds.

To find the $\bE$-triangle required for condition \ref{E-triangle-CY}, consider the following triangles for $c\in \sC$,
\begin{equation} \label{approx-triangles}
b \to \Sigma a\to c \to \Sigma b
\quad \text{and} \quad
a'\to \Sigma a \to  b'\to \Sigma a',
\end{equation}
where $a\in\sA$, $b\in\sB$, $a'\in\sA'$ and $b'\in\sB'$.
Applying the octahedral axiom, we get:
\[
\xymatrix@!R=5px{
&b \ar@{=}[r] \ar[d] &b \ar[d] &\\
a'\ar[r] \ar@{=}[d] &\Sigma a \ar[r] \ar[d]&b' \ar[r] \ar[d]&\Sigma a'  \ar@{=}[d]\\
a' \ar[r]&c \ar[r] \ar[d]&y \ar[r] \ar[d]& \Sigma a' \\
&\Sigma b \ar@{=}[r]&\Sigma b&
}
\]
We first observe that $y \in \cY$.
In the triangle $b'\to y \to \Sigma b\to \Sigma b'$ the outer terms $b'\in \sB'$ and $\Sigma b\in \Sigma \sB \subseteq \sB'$, so $y\in \sB'$. In the triangle $c\to y \to \Sigma a'\to \Sigma c$ the outer terms $c\in \Sigma^2 \sA$ and $\Sigma a'\in \Sigma \sA' \subseteq \Sigma^2 \sA$, so $y\in \Sigma^2 \sA$ and thus $y\in \cY$. 

Since $\Sigma b\in \sB$ and $\Sigma c \in   \sB$ we get that $\Sigma^2 a \in \sB$.  In the triangle $b'\to \Sigma a' \to \Sigma^2 a\to \Sigma b'$ the outer terms $b'\in \sB' \subseteq \sB$ and $\Sigma^2 a\in  \sB$, so $\Sigma a'\in \Sigma \sA' \cap \sB = \cX$. Thus, the triangle 
\[
c\to y \to \Sigma a' \to \Sigma c
\]
gives the $\bE$-triangle required for condition \ref{E-triangle-CY} of Definition \ref{DefCotorsion}. 

To find the $\bE$-triangle required for condition \ref{E-triangle-XC}, consider $c \in \sC$ as above and a triangle 
\[
a''\to  b \to  b'' \to \Sigma a''
\]
where $b$ is the object in the triangle in \eqref{approx-triangles}, $a''\in\sA'$ and $b''\in\sB'$.
Observe that $\Sigma^{-1}c\in \Sigma \sA$ and $\Sigma a\in \Sigma \sA$, so $b\in \Sigma \sA \subseteq \Sigma^2 \sA$. Since $\Sigma a'' \in \Sigma \sA' \subseteq \Sigma^2 \sA$, we get $b''\in \Sigma^2 \sA$.
Applying the octahedral axiom again, we get:
\[
\xymatrix@!R=5px{
&a'' \ar@{=}[r] \ar[d] &a'' \ar[d] &\\
\Sigma^{-1}c\ar[r] \ar@{=}[d] &b \ar[r] \ar[d]&\Sigma a \ar[r] \ar[d]&c  \ar@{=}[d]\\
\Sigma^{-1}c \ar[r]&b'' \ar[r] \ar[d]&x \ar[r] \ar[d]& c \\
&\Sigma a'' \ar@{=}[r]&\Sigma a''&
}
\]
Clearly $x\in  \cX$ and $b''\in \cY$, so the triangle 
\[
b''\to x\to c \to \Sigma b''
\]
gives the $\bE$-triangle required for condition \ref{E-triangle-XC} of Definition~\ref{DefCotorsion}.
Thus the assignment $\phi \colon (\sA',\sB') \mapsto (\cX,\cY)$ defines a map from co-t-structures $(\sA',\sB')$ such that $\sA \subseteq \sA' \subseteq \Sigma \sA$ to complete cotorsion pairs in $\sC$.

\textbf{Step 2:} We now construct the map in the other direction. Let $(\sA,\sB)$ be a co-t-structure in $\sT$ and let $(\cX,\cY)$ be a complete cotorsion pair in the extended coheart $\sC=\Sigma^2 A \cap B$. Consider the following pair of subcategories of $\sT$:
\[
(\sA',\sB'):=(\add(\Sigma^{-1} \sA * \Sigma^{-1}\cX), \add(\cY * \Sigma^2 \sB)).
\]
The subcategories are clearly orthogonal.  

To see that $\Sigma^{-1} \sA' \subseteq \sA'$ and $\Sigma \sB' \subseteq \sB'$ we observe that $\Sigma^{-1} \sA * \Sigma^{-1} \cX = \sA * \Sigma^{-1} \cX$ and $\cY * \Sigma^2 \sB = \cY * \Sigma \sB$. We show the first equality holds; the second equality is analogous. The inclusion $\Sigma^{-1} \sA * \Sigma^{-1} \cX \subseteq \sA * \Sigma^{-1} \cX$ is immediate because $\Sigma^{-1} \sA \subseteq \sA$. For the other inclusion, consider a decomposition of $t \in \sA * \Sigma^{-1} \cX$: $a \to t \to \Sigma^{-1} x \to \Sigma a$ with $a \in \sA$ and $x \in \cX$. Decompose $a$ with respect to the co-t-structure $(\Sigma^{-1} \sA,\Sigma^{-1} \sB)$ to get a triangle $\Sigma^{-1} a' \to a \to \Sigma^{-1} s \to a'$ with $s \in \sS = \Sigma \sA \cap \sB$. Since $(\cX,\cY)$ is a complete cotorsion pair, $\sS \subseteq \cX$. Applying the octahedral axiom to the two triangles gives
\[
\xymatrix@!R=5px{
                                                & \Sigma^{-1} a' \ar[d] \ar@{=}[r] & \Sigma^{-1} a' \ar[d]          & \\
\Sigma^{-2} x \ar@{=}[d] \ar[r] & a \ar[r] \ar[d]                             & t \ar[r] \ar[d]                      & \Sigma^{-1} x \ar@{=}[d] \\
\Sigma^{-2} x \ar[r]                  & \Sigma^{-1} s \ar[r] \ar[d]           & \Sigma^{-1} x' \ar[r] \ar[d] & \Sigma^{-1} x \\
                                                & a' \ar@{=}[r]                               & a'                                      & 
}
\]
in which $x' \in \cX$, giving a decomposition of $t \in \Sigma^{-1} \sA * \Sigma^{-1} \cX$. Hence $\sA * \Sigma^{-1} \cX = \Sigma^{-1} \sA * \Sigma^{-1} \cX$.
The condition $\sA\subseteq \sA'\subseteq \Sigma \sA$ also holds.

It remains for us to construct the approximation triangle from the definition of the co-t-structure. 
Consider the following triangles for $t\in \sT$:
\[
a_t \to t \to b_t \to \Sigma a_t
\quad \text{and} \quad
\Sigma b \to \Sigma^2 a\to b_t \to  \Sigma^2 b,
\]
where $a_t, a \in\sA$ and $b,b_t \in \sB$.
Since $\Sigma^2 a \in \Sigma^2 \sA \cap \sB=\sC$, there is a triangle
\[
\Sigma^{-1}x \to \Sigma^2 a \to y \to x
\]
coming from the $\bE$-triangle occurring in condition \ref{E-triangle-CY} of the definition of complete cotorsion pair.
Applying the octahedral axiom twice, we get:
\[
\xymatrix@!R=5px{
&\Sigma b \ar@{=}[r] \ar[d] &\Sigma b \ar[d] &\\
\Sigma^{-1} x \ar[r] \ar@{=}[d] &\Sigma^{2} a \ar[r] \ar[d]&y \ar[r] \ar[d]&x  \ar@{=}[d]\\
\Sigma^{-1} x \ar[r]&b_t \ar[r] \ar[d]&b' \ar[r] \ar[d]& x \\
& \Sigma^{2} b \ar@{=}[r]&\Sigma^{2} b&
}
\qquad
\xymatrix@!R=5px{
&\Sigma^{-1} x \ar@{=}[r] \ar[d] &\Sigma^{-1} x \ar[d] &\\
t \ar[r] \ar@{=}[d] & b_t \ar[r] \ar[d]&\Sigma a_t \ar[r] \ar[d]&\Sigma t  \ar@{=}[d]\\
t \ar[r]&b' \ar[r] \ar[d]&\Sigma a' \ar[r] \ar[d]& \Sigma t \\
& x \ar@{=}[r]&x &
}
\]
where, in the left-hand diagram, we see $b'\in \cY * \Sigma^{2} \sB \subseteq \sB'$,
and in the right-hand diagram, we have $a'\in \sA * \Sigma^{-1} \cX \subseteq \sA'$. Thus the triangle
\[
a' \to t\to b' \to \Sigma a'
\]
is an approximation triangle for $t$ with respect to the co-t-structure $(\sA',\sB')$.
Thus the assignment $\psi \colon (\cX,\cY) \mapsto (\sA'= \add(\Sigma^{-1} \sA * \Sigma^{-1} \cX),\sB' = \add(\cY * \Sigma^2 \sB))$ defines a map from complete cotorsion pairs in $\sC$ to co-t-structures $(\sA',\sB')$ such that $\sA \subseteq \sA' \subseteq \Sigma \sA$.

\textbf{Step 3:}  We now show that the maps $\phi$ and $\psi$ defined in Steps 1 and 2 are mutually inverse.
Let $(\sA,\sB)$ be a co-t-structure in $\sT$, let $(\cX,\cY)$ be a complete cotorsion pair in the extended coheart $\sC=\Sigma^2 A \cap B$ and let $(\sA',\sB') = \psi\big( (\cX,\cY) \big)$ be the co-t-structure constructed in Step 2. 
Let $(\cX',\cY') = \phi \big( (\sA',\sB') \big)$ be the complete cotorsion pair constructed from $(\sA',\sB')$ is Step 1. That is,
\[
\cX':=\sB\cap \Sigma \sA'\subseteq \sC 
\quad \text{and} \quad 
\cY':=\sB'\cap \Sigma^2 \sA \subseteq \sC.
\]
Since $\Sigma \cY'\subseteq \add(\Sigma \cY * \Sigma^2 \sB) $ and $\cX\subseteq \Sigma^2 \sA \cap \sB$ we get that $\sT(\cX,\Sigma\cY')=0$. 
For an object $y'\in \cY'$ we can consider the triangle $y'\to y \to x \to \Sigma y'$ coming from the definition of the complete cotorsion pair $(\cX,\cY)$. Since the map $x \to \Sigma y'$ is zero, the triangle splits and $y'$ is a summand of $y$. Since $\cY$ is closed under summands, we get $\cY'\subseteq \cY$. Similarly $\sT(\cX',\Sigma\cY)=0$. The splitting of the  triangle $y\to y' \to x' \to \Sigma y$ from the definition of the complete cotorsion pair $(\cX',\cY')$ gives that $\cY\subseteq \cY'$. Since $\cX=\sC\cap(^{\perp}\Sigma\cY)=\sC\cap(^{\perp}\Sigma\cY')=\cX'$, the cotorsion pairs coincide and  $\phi\psi=\mathrm{id}$.

Let $(\sA',\sB')$ be a co-t-structure such that $\sA\subseteq \sA'\subseteq \Sigma \sA$ (equivalently, $\Sigma \sB\subseteq \sB'\subseteq  \sB$). 
Consider the co-t-structure $(\sA'',\sB''):=(\add(\Sigma^{-1} \sA * (\Sigma^{-1}\sB \cap \sA' )), \add((\sB'\cap \Sigma^2 A) * \Sigma^2 \sB))$. Clearly $\sA''\subseteq \sA'$ and $\sB''\subseteq \sB'$ and since both pairs of subcategories are co-t-structures we get $(\sA',\sB')=(\sA'',\sB'')$ and $\psi\phi=\mathrm{id}$. Thus we get the desired bijection.
\end{proof}

\begin{remark}
In the definition of the intermediate co-t-structure $(\sA',\sB')$ obtained from a complete cotorsion pair $(\cX,\cY)$ in the extended coheart $\sC = \sS * \Sigma \sS$ in Theorem~\ref{thm:co-hrs} it is not obvious that $\Sigma^{-1} \sA * \Sigma^{-1} \cX$ and $\cY * \Sigma^2 \sB$ are closed under summands, hence we are required to take the additive closure.  
However, one can check that $\sT(\sA,\cX) = 0$ and $\sT(\cY,\Sigma^2 \sB) = 0$, so that in the case that $\sT$ is Krull-Schmidt, applying \cite[Proposition 2.1]{IY}, we see that $\Sigma^{-1} \sA * \Sigma^{-1} \cX=\sA * \Sigma^{-1} \cX$ and $\cY * \Sigma^2 \sB$ are closed under direct summands.

In light of Step 2 of the proof of Theorem~\ref{thm:co-hrs}, there are alternative descriptions of $\sA'$ and $\sB'$. The description we have chosen has two advantages: it is the closest parallel to classic HRS tilting for t-structures using torsion pairs, and it makes the equality $\Sigma^{-1} \cX * \cY = \Sigma^{-1} \sS * \sS * \Sigma \sS$ intuitive; see Figure~\ref{fig:schematic} for a schematic of the situation. Note that the equality $\Sigma^{-1} \cX * \cY = \Sigma^{-1} \sS * \sS * \Sigma \sS$ holds, since $\sT(\Sigma^{-1} \cX ,\Sigma \cY) =0$, so $\Sigma^{-1} \cX * \cY$ is extension closed \cite[Lemma 8]{NSZ}. 

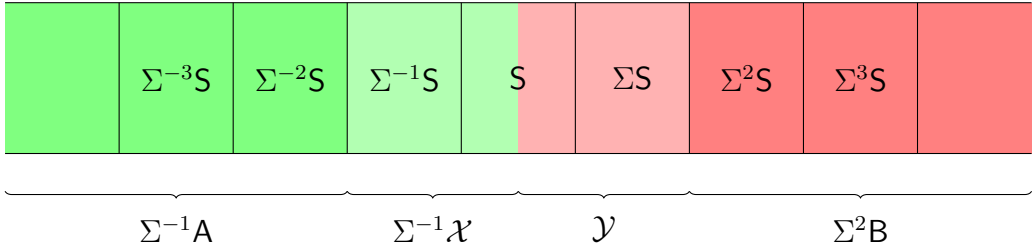
\begin{figure}
\begin{center}
\begin{tikzpicture}

\fill[green!50] (-6,0) -- (-1.5,0) -- (-1.5,2) -- (-6,2) -- cycle;
\fill[red!50] (3,0) -- (7.5,0) -- (7.5,2) -- (3,2) -- cycle;

\fill[green!30] (-1.5,0) -- (0.75,0) -- (0.75,2) -- (-1.5,2) -- cycle;
\fill[red!30] (0.75,0) -- (3,0) -- (3,2) -- (0.75,2) -- cycle;

\draw[decoration={brace,mirror},decorate] (3,-0.5) --  (7.5,-0.5);
\draw[decoration={brace,mirror},decorate] (-6,-0.5) -- (-1.5,-0.5);
\draw[decoration={brace,mirror},decorate] (-1.5,-0.5) -- (0.75,-0.5);
\draw[decoration={brace,mirror},decorate] (0.75,-0.5) -- (3,-0.5);

\draw (-6,2) -- (7.5,2);
\draw (-6,0) -- (7.5,0);

\draw (-4.5,0) -- (-4.5,2);
\draw (-3,0) -- (-3,2);
\draw (-1.5,0) -- (-1.5,2);
\draw (0,0) -- (0,2);
\draw (1.5,0) -- (1.5,2);
\draw(3,0) -- (3,2);
\draw (4.5,0) -- (4.5,2);
\draw (6,0) -- (6,2);

\node at (-3.75,1) {$\Sigma^{-3} \sS$};
\node at (-2.25,1) {$\Sigma^{-2} \sS$};
\node at (-0.75,1) {$\Sigma^{-1} \sS$};
\node at (0.75,1) {$\sS$};
\node at (2.25,1) {$\Sigma \sS$};
\node at (3.75,1) {$\Sigma^2 \sS$};
\node at (5.25,1) {$\Sigma^3 \sS$};
\node at (-3.75,-1) {$\Sigma^{-1} \sA$};
\node at (-0.375,-1) {$\Sigma^{-1} \cX$};
\node at (1.875,-1) {$\cY$};
\node at (5.25,-1) {$\Sigma^2 \sB$};

\end{tikzpicture}
\end{center}
\caption{Schematic showing the construction of the intermediate co-t-structure $(\sA',\sB')$ in Theorem~\ref{thm:co-hrs} from a complete cotorsion pair $(\cX,\cY)$ in the extended coheart $\sC= \sS* \Sigma \sS$ of the co-t-structure $(\sA,\sB)$.} \label{fig:schematic}
\end{figure}
\end{remark}

The corollary below shows that Theorem~\ref{thm:co-hrs} recovers the bijection between co-t-structures intermediate with respect to $(\sA,\sB)$ and silting subcategories $\sS' \subseteq \sS * \Sigma \sS$ in \cite[Theorem 2.3]{IJY}.

\begin{corollary} \label{cor:co-hrs}
Suppose 
$(\sA,\sB)$ is a co-t-structure in $\sT$, and $\sC = \Sigma^2 \sA \cap \sB$ is the extended coheart of $(\sA,\sB)$.
If $(\cX,\cY)$ is a complete cotorsion pair in $\sC$, then its core $\cW = \cX \cap \cY$ is the coheart of the corresponding intermediate co-t-structure $(\sA',\sB') = (\add(\Sigma^{-1} \sA* \Sigma^{-1}\cX),\add(\cY*\Sigma^2 \sB))$.
\end{corollary}

\begin{proof}
Let $\sS' = \Sigma \sA' \cap \sB'$ be the coheart of the co-t-structure $(\sA',\sB')$. By Theorem~\ref{thm:co-hrs}, $\cX = \sB \cap \Sigma \sA'$ and $\cY = \sB' \cap \Sigma^2 \sA$. 
Hence $\cW = \cX \cap \cY = \sS' \cap \sC$. But since $\sS' = \Sigma \sA' \cap \sB' \subseteq \Sigma^2 \sA \cap \sB = \sC$, we have $\cW = \sS'$.
\end{proof}

We finish this section with a straightforward example illustrating Theorem~\ref{thm:co-hrs}.

\begin{example}
Let $\kk$ be a field and let $A_3$ be the equi-oriented Dynkin diagram of type $A_3$. In the diagram below we show the indecomposable objects in the Auslander--Reiten quiver of $\sT = \Db(\kk A_3)$; note that we suppress the arrows in the AR quiver. The diagram depicts a co-t-structure $(\sA,\sB)$, its extended coheart $\sC = \Sigma^2 \sA \cap \sB$, an intermediate co-t-structure $(\sA',\sB')$ such that $\sA \subseteq \sA' \subseteq \Sigma \sA$, and a complete cotorsion pair $(\cX,\cY) = (\sB \cap \Sigma \sA', \sB' \cap \Sigma^2 \sA)$.
We highlight these objects in the diagram as follows.

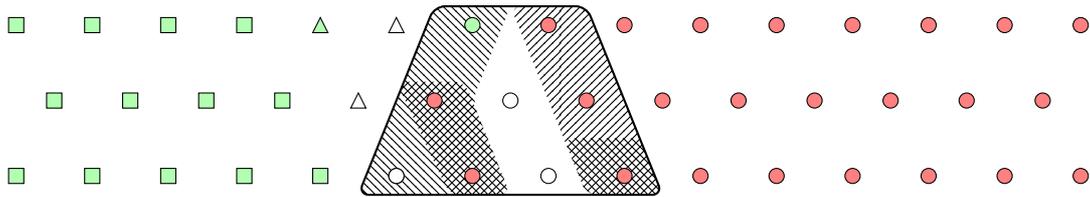
\begin{center}
\begin{tikzpicture}

\fill[rounded corners,pattern=north west lines] (4.5,-0.25) -- (6.5,-0.25) -- (6, 1) -- (6.5, 2.25) -- (5.5, 2.25) -- cycle;
\fill[rounded corners,pattern=north west lines] (7.2, 0.5) -- (8.2, 0.5) -- (8.5,-0.25) -- (7.5,-0.25) -- cycle; 


\fill[rounded corners,pattern=north east lines] (5.7,-0.25) -- (6.5,-0.25) -- (5.9, 1.25)  -- (5, 1.25) -- cycle;
\fill[rounded corners,pattern=north east lines] (6.5, 2.25) -- (7.5, 2.25) -- (8.5,-0.25) -- (7.5,-0.25) -- cycle; 

\foreach \x in {0,1,...,4} \foreach \y in {0}
  { \fill[green!30] (\x - 0.1, \y + 0.1) -- (\x + 0.1, \y + 0.1) -- (\x + 0.1, \y - 0.1) -- (\x - 0.1, \y - 0.1) -- cycle; }

\foreach \x in {0.5,1.5,...,3.5} \foreach \y in {1}
  { \fill[green!30] (\x - 0.1, \y + 0.1) -- (\x + 0.1, \y + 0.1) -- (\x + 0.1, \y - 0.1) -- (\x - 0.1, \y - 0.1) -- cycle; }

\foreach \x in {0,1,...,3} \foreach \y in {2}
  { \fill[green!30] (\x - 0.1, \y + 0.1) -- (\x + 0.1, \y + 0.1) -- (\x + 0.1, \y - 0.1) -- (\x - 0.1, \y - 0.1) -- cycle; }
\fill[green!30] (4, 2 + 0.1) -- (4 + 0.1, 2 - 0.1) -- (4 - 0.1, 2 - 0.1) -- cycle;
\fill[green!30] (6,2) circle (1mm);

\foreach \x in {8,9,...,14} \foreach \y in {0}
 { \fill[red!50] (\x, \y) circle (1mm); }
\fill[red!50] (6,0) circle (1mm);

\foreach \x in {7.5,8.5,...,13.5} \foreach \y in {1}
  { \fill[red!50] (\x, \y) circle (1mm); }
\fill[red!50] (5.5,1) circle (1mm);  

\foreach \x in {7,8,...,14} \foreach \y in {2}
  { \fill[red!50] (\x, \y) circle (1mm); }  

\foreach \x in {0,1,...,4} \foreach \y in {0}
  { \draw (\x - 0.1, \y + 0.1) -- (\x + 0.1, \y + 0.1) -- (\x + 0.1, \y - 0.1) -- (\x - 0.1, \y - 0.1) -- cycle; }

\foreach \x in {0.5,1.5,...,3.5} \foreach \y in {1}
  { \draw (\x - 0.1, \y + 0.1) -- (\x + 0.1, \y + 0.1) -- (\x + 0.1, \y - 0.1) -- (\x - 0.1, \y - 0.1) -- cycle; }

\foreach \x in {0,1,...,3} \foreach \y in {2}
  { \draw (\x - 0.1, \y + 0.1) -- (\x + 0.1, \y + 0.1) -- (\x + 0.1, \y - 0.1) -- (\x - 0.1, \y - 0.1) -- cycle; }  

\foreach \x in {5,6,...,14} \foreach \y in {0}
 { \draw (\x, \y) circle (1mm); }

\foreach \x in {5.5,6.5,7.5,...,13.5} \foreach \y in {1}
  { \draw (\x, \y) circle (1mm); }

\foreach \x in {6,7,...,14} \foreach \y in {2}
  { \draw (\x, \y) circle (1mm); }

\draw (5, 2 + 0.1) -- (5 + 0.1, 2 - 0.1) -- (5 - 0.1, 2 - 0.1) -- cycle;
\draw (4.5, 1 + 0.1) -- (4.5 + 0.1, 1 - 0.1) -- (4.5 - 0.1, 1 - 0.1) -- cycle; 
\draw (4, 2 + 0.1) -- (4 + 0.1, 2 - 0.1) -- (4 - 0.1, 2 - 0.1) -- cycle; 
 
\draw[rounded corners,thick] (5.5, 2.25) -- (7.5, 2.25) -- (8.5,-0.25) -- (4.5,-0.25) -- cycle;
\draw[fill=white] (5,0) circle (1mm);
\end{tikzpicture}
\end{center}

\noindent
\begin{center}
\begin{tabular}{@{} p{0.15\textwidth} @{} p{0.73\textwidth} @{}}
symbols                                                     & indecomposable objects of \\ \midrule
\tikz{\draw (0.1, 0.1) rectangle (-0.1,-0.1)} & the aisle $\sA$ of the co-t-structure $(\sA,\sB)$ \\
\tikz{\draw (0,0) circle (1mm)}                    & the co-aisle $\sB$ of the co-t-structure $(\sA,\sB)$ \\
\tikz{\draw (7, 2 + 0.1) -- (7 + 0.1, 2 - 0.1) -- (7 - 0.1, 2 - 0.1) -- cycle} & neither the aisle $\sA$ nor the co-aisle $\sB$ of the co-t-structure $(\sA,\sB)$ \\
\tikz{\draw[fill=green!30] (0.1, 0.1) rectangle (-0.1,-0.1)},  \tikz{\draw[fill=green!30] (7, 2 + 0.1) -- (7 + 0.1, 2 - 0.1) -- (7 - 0.1, 2 - 0.1) -- cycle} or
\tikz{\draw[fill=green!30] (0,0) circle (1mm)}   & the aisle $\sA'$ of the co-t-structure $(\sA',\sB')$ \\
\tikz{\draw[fill=red!30] (0,0) circle (1mm)}    & the co-aisle $\sB'$ of the co-t-structure $(\sA',\sB')$ \\
\tikz{\fill[pattern=north west lines] (-0.25,-0.25) rectangle (0.25,0.25); \draw[fill=white] (0,0) circle (1mm)} & the cotorsion class $\cX$ of the complete cotorsion pair $(\cX,\cY)$ \\
\tikz{\fill[pattern=north east lines] (-0.25,-0.25) rectangle (0.25,0.25); \draw[fill=white] (0,0) circle (1mm)} & the cotorsionfree class $\cY$ of the complete cotorsion pair $(\cX,\cY)$ \\
\multicolumn{2}{@{} l}{The extended coheart $\sC$ is indicated by the outlined region:
                    \scalebox{0.25}{\tikz{\draw[rounded corners,very thick] (5.5, 2.25) -- (7.5, 2.25) -- (8.5,-0.25) -- (4.5,-0.25) -- cycle;}}}
\end{tabular}
\end{center}
\end{example}

\section{Cotorsion pairs versus torsion pairs} \label{sec:cotorsion-torsion}

The aim of this section is to provide a direct proof of Theorem~\ref{intro:cotorsion-torsion}. We restrict to the following setup so that we can apply the setup of Section~\ref{sec:yoneda}.

\begin{setup} \label{setup}
From now on we will assume that $\sT$ is essentially small, Hom-finite, $\kk$-linear and Krull-Schmidt, where $\kk$ is a commutative noetherian ring. Note that in that case $\sT$ is automatically idempotent complete.
Let $\sS$ be a presilting subcategory of $\sT$ such that $\mod \sS$ is noetherian and set $\sC = \sS* \Sigma \sS$.
Note that, since $\sS$ is silting in $\thick \sS$, the subcategory $ \sS * \Sigma \sS$ is closed under summands and extensions.
\end{setup}

Under the assumptions of Setup~\ref{setup}, if $\sS = \add(s)$, then $\mod \sS \simeq \mod E$, where $E = \sT(s,s)$ is a noetherian ring, making $\mod \sS$ noetherian.

\begin{proposition} \label{prop:cotorsion-to-torsion}
Suppose that the hypotheses of Setup~\ref{setup} hold. Then the equivalence $F \colon \sC/\Sigma \sS \to \mod \sS$ induces a well-defined map
\begin{align*}
\Phi \colon \{ \text{cotorsion pairs in } \sC\} & \to \{ \text{torsion pairs in } \mod \sS\}, \\
(\cX,\cY)                                                  & \mapsto (\cT = F\cY, \cF = \cT\orth)
\end{align*}
which restricts to a well-defined map
\[
\Phi \colon 
\{\text{complete cotorsion pairs in } \sC \} \to
\{\text{functorially finite torsion pairs in } \mod \sS \}.
\]
\end{proposition}

\begin{proof}
Let $(\cX,\cY)$ be a cotorsion pair in $\sC$. We claim that the essential image $\cT = F\cY$ is a torsion class in $\mod \sS$. Since $\mod \sS$ is  noetherian, by \cite[Proposition 3.5]{Liu-Stanley} it is enough to show that $\cT$ is closed under quotients and extensions.

We start by showing that $\cT = F \cY$ is closed under quotients. Consider an exact sequence $t \rightlabel{\phi} u \to 0$ in $\mod \sS$ with $t \in \cT$. Lifting this to  
$\sC$ via $F$, there are objects $y\in \cY$ and $v\in \sC$ and a morphism $f \colon y \to v$ such that $Fy = t$, $Fv = u$ and $Ff = \phi$. Completing the morphism $f$ to a distinguished triangle in $\sT$ gives
\[
c \too y \rightlabel{f} v \rightlabel{g} \Sigma c.
\]
Applying $F$ to this triangle, we get the exact sequence
\[
\sT(-,y)|_\sS \rightlabel{\sT(-,f)|_\sS} \sT(-,v)|_\sS  \rightlabel{\sT(-,g)|_\sS} \sT(-,\Sigma c)|_\sS \too \sT(-,\Sigma y)|_\sS.
\]
Since $\sT(-,f)|_\sS = \phi$ is an epimorphism, we have $\sT(-,g)|_\sS = 0$. Moreover, $\Sigma y \in \Sigma \sS * \Sigma^2 \sS$ so that $\sS$ presilting implies that $\sT(-,\Sigma y)|_\sS=0$. 
Hence, $\sT(-,\Sigma c)|_\sS = 0$. In particular, it follows that $\Sigma c \in (\sS * \Sigma \sS * \Sigma^2 \sS) \cap \sS\orth$, in which case we get that $c \in \sS * \Sigma \sS$. 
Now applying $\sT(\cX,-)$ to the triangle above gives $\sT(\cX,\Sigma v) = 0$, which means $v \in \cY$ because $(\cX,\cY)$ is a cotorsion pair. Hence, $u \simeq Fv \in \cT$ and $\cT$ is closed under quotients. 

Next we show that $\cT$ is closed under extensions. Consider a short exact sequence 
\[
0 \too t' \rightlabel{\phi} t \too t'' \too 0
\]
in $\mod \sS$ with $t',t'' \in \cT$. Lift the morphism $\phi \colon t' \to t$ to $\sC$ to obtain a morphism $f \colon y' \to y$ such that $y'\in\cY$, $Ff = \phi$, $Fy' = t'$ and $Fy = t$. Extend $f$ to a distinguished triangle to get
\[
y' \rightlabel{f} y \too c \too \Sigma y'.
\]
Applying the restricted Yoneda functor to this triangle and noting that $\sT(-,\Sigma y')|_\sS = 0$ gives a commutative diagram,
\[
\xymatrix@!R=5px{
Fy' \ar[r]^-{Ff} \ar@{=}[d] & Fy \ar[r] \ar@{=}[d] & Fc \ar[r] & 0, \\
t' \ar[r]_-{\phi}                  & t \ar[r]                     & t'' \ar[r] & 0                  
}
\]
whence $Fc \simeq t''\simeq Fy''$ for some $y''\in \cY$. 

If $c \in \cY \subseteq \sC$, then we are done. However, we do not know this to be the case. Reading off from the triangle above, $c \in \sS * \Sigma \sS * \Sigma^2 \sS$, so we can consider a decomposition $m \to c \to n \to \Sigma n$ in which $m \in \sS * \Sigma \sS$ and $n \in \Sigma^2 \sS$. 
Consider the octahedral diagram obtained from the two triangles below.
\[
\xymatrix@!R=5px{
                             & m \ar@{=}[r] \ar[d]   & m \ar[d]                     & \\
y \ar[r] \ar@{=}[d] & c  \ar[r] \ar[d]            & \Sigma y' \ar[r] \ar[d] & \Sigma y \ar@{=}[d] \\
y \ar[r]                  & n \ar[r] \ar[d]             & \Sigma e \ar[r] \ar[d]  & \Sigma y \\
                            & \Sigma m \ar@{=}[r] & \Sigma m                   &
}
\]
Applying the restricted Yoneda functor to a rotation of the lower horizontal triangle gives an isomorphism $Fe \rightiso Fy$.
Rotating the right-hand vertical triangle gives
\[
y' \to e \to m \to \Sigma y',
\]
showing that $e \in \sC$.
Consider the octahedral diagram obtained using this triangle together with the triangle $e \to y \to n \to \Sigma e$.
\[
\xymatrix@!R=5px{
                             & \Sigma^{-1} n \ar@{=}[r] \ar[d]   & \Sigma^{-1} n \ar[d]   & \\
y' \ar[r] \ar@{=}[d] & e  \ar[r] \ar[d]                             & m \ar[r] \ar[d]             & \Sigma y' \ar@{=}[d] \\
y' \ar[r]                  & y \ar[r] \ar[d]                               & c \ar[r] \ar[d]              & \Sigma y' \\
                            & n \ar@{=}[r]                                & n                                &
}
\]
Note that from the previous diagram $c$ is indeed isomorphic to the cone of the map $y'\rightarrow y$. Applying the restricted Yoneda functor to the two horizontal triangles gives,
\[
\xymatrix@!R=5px{
Fy' \ar[r] \ar@{=}[d] & Fe \ar[r] \ar[d]^-{\sim} & Fm \ar[r] \ar[d]^-{\sim} & 0 \\
Fy' \ar[r]                  & Fy \ar[r]                       & Fc \ar[r]                       & 0.
}
\]
In particular, there is an isomorphism $Fy'' \rightiso Fm$, so by the argument showing $\cT$ is closed under quotients, we see that $m$ lies in $\cY$. Since $\cY$ is closed under extensions, it follows that $e \in \cY$. Hence $Fe \simeq Fy \simeq t$, showing that $t \in \cT$, as required.

Finally, we check that the map induced by the restricted Yoneda functor restricts as claimed. We need to show that $\cT = F\cY$ is covariantly finite when $(\cX,\cY)$ is a complete cotorsion pair.
Let $m \in \mod \sS$. Suppose $c \in \sC$ is such that $Fc \simeq m$. Consider a decomposition triangle of $c$ with respect to the complete cotorsion pair $(\cX,\cY)$,
\[
c \rightlabel{f} y \too x \too \Sigma c.
\]
Note that $f \colon c \to y$ is a left $\cY$-approximation of $c$ in $\sC$. 
We claim that $Ff \colon m \to Fy$ is a left $\cT$-approximation of $m$ in $\mod \sS$. Suppose $\phi \colon m \to t$ is a morphism in $\mod \sS$ with $t \in \cT$. Then there exist $y' \in \cY$ and $g \colon c \to y'$ such that $Fy' \simeq t$ and $Fg \simeq \phi$. Since $f$ is a left $\cY$-approximation of $c$ in $\sC$, there exists $h \colon y \to y'$ such that $g = hf$. Applying $F$ to this composition gives $\phi \simeq Fg = Fh Ff$, that is $Ff$ is a left $\cT$-approximation, as required. 
\end{proof}

The next lemma provides a useful criterion to detect when a cotorsion pair is complete.

\begin{lemma} \label{lem:complete}
Let $\sS$ be a presilting subcategory of a triangulated category $\sT$ and $\sC = \sS * \Sigma \sS$. 
Suppose $(\cX,\cY)$ is a cotorsion pair in $\sC$, then $(\cX,\cY)$ is a complete cotorsion pair in $\sC$ if and only if for each object $s$ of $\sS$ there exists a triangle $s \to y \to x \to \Sigma s$ with $x \in \cX$ and $y \in \cY$.
\end{lemma}

\begin{proof}
If $(\cX,\cY)$ is a complete cotorsion pair in $\sC$ then by definition such a triangle exists for $s \in \sS$ because such a triangle exists for each $c \in \sC$.

Conversely, suppose $(\cX,\cY)$ is a cotorsion pair in $\sC$ and that for each $s \in \sS$ there is a triangle $s \to y \to x \to \Sigma s$ with $x \in \cX$ and $y \in \cY$. Let $c \in \sC$ and take a decomposition triangle, $s_2 \to s_1 \to c \to \Sigma s_2$, and the triangle $s_2 \to y_2 \to x_2 \to \Sigma s_2$ given by the assumption. Applying the octahedral axiom to these triangles gives the following diagram.
\[
\xymatrix@!R=5px{
                                                & \Sigma^{-1} x_2 \ar@{=}[r] \ar[d]  & \Sigma^{-1} x_2 \ar[d] & \\
\Sigma^{-1} c \ar[r] \ar@{=}[d] & s_2 \ar[r] \ar[d]                              & s_1 \ar[r] \ar[d]            & c \ar@{=}[d] \\
\Sigma^{-1} c \ar[r]                  & y_2 \ar[r] \ar[d]                              & e \ar[r] \ar[d]                 & c \\
                                                & x_2 \ar@{=}[r]                               & x_2                              &
}
\]
We have $e \in \cX$ since $\sS \subseteq \cX$, because $(\cX,\cY)$ is a cotorsion pair. Thus, the triangle $y_2 \to e \to c \to \Sigma y_2$ provides the second triangle required for completeness in Definition~\ref{DefCotorsion}.

To obtain the first triangle in Definition~\ref{DefCotorsion}, we use the octahedral axiom again together with the triangle $s_1 \to y_1 \to x_1 \to \Sigma s_1$ given by the assumption:
\[
\xymatrix@!R=5px{
                                 & \Sigma^{-1} x_1 \ar@{=}[r] \ar[d]  & \Sigma^{-1} x_1 \ar[d] & \\
s_2 \ar[r] \ar@{=}[d] & s_1 \ar[r] \ar[d]                              & c \ar[r] \ar[d]                & \Sigma s_2 \ar@{=}[d] \\
s_2 \ar[r]                  & y_1 \ar[r] \ar[d]                              & d \ar[r] \ar[d]                & \Sigma s_2 \\
                                & x_1 \ar@{=}[r]                               & x_1                              &
}
\]
Analogously, we have $d \in \cY$ because $\Sigma \sS \subseteq \cY$, making $c \to d \to x_1 \to \Sigma c$ the required triangle.
\end{proof}

\begin{remark} \label{rem:complete}
We make two observations regarding Lemma~\ref{lem:complete}.
\begin{enumerate}
\item \label{rem:complete-1} Lemma~\ref{lem:complete}, in fact, holds in the case that $(\cX,\cY)$ is an Ext-orthogonal pair of subcategories of $\sC$ closed under extensiona and direct summands such that $\sS \subseteq \cX$ and $\Sigma \sS \subseteq \cY$.
\item Let $\sS$, $\sC$ and $(\cX,\cY)$ be as in the statement of Lemma~\ref{lem:complete}.
In the triangle $s \to y \to x \to \Sigma s$ we observe that since $\sS \subseteq \cX$ and $\Sigma \sS \subseteq \cY$, we have $x,y \in \cX \cap \cY$. In the context of Section \ref{sec:co-hrs} this provides a decomposition of $\sS$ in $\Sigma^{-1}\sS' * \sS'$, where $\sS'=\cX \cap \cY$ is the coheart of the co-t-structure $(\sA',\sB')$.
\end{enumerate}
\end{remark}

We now define an inverse to the restricted map in Proposition~\ref{prop:cotorsion-to-torsion}.

\begin{proposition} \label{prop:torsion-to-cotorsion}
Let $\sS$ be a presilting subcategory of $\sT$ and $\sC = \sS * \Sigma \sS$. There is a well-defined map
\[
\Theta \colon \{ \text{functorially finite torsion pairs in } \mod \sS \}  \to \{ \text{complete cotorsion pairs in } \sC\}. 
\]
\end{proposition}

\begin{proof}
Let $(\cT,\cF)$ be a functorially finite torsion pair in $\mod \sS$, $\cY = \{c \in \sC \mid Fc \in \cT\}$ and $\cX = {}\orth (\Sigma \cY)\cap \sC$.
The subcategories $\cY$ and $\cX$ are closed under direct summands, since $\cT$ is closed under direct summands and $\cX$ is defined as an orthogonal.
We set $\Theta\big( (\cT,\cF) \big) = (\cX,\cY)$.
We claim that $(\cX,\cY)$ is a complete cotorsion pair.
We start by showing that $\cY$ is closed under extensions (in $\sC$).

Let $y' \rightlabel{f} y \too y'' \too \Sigma y$ be an extension with $y', y'' \in \cY$. Applying $F$ to this $\bE$-triangle in $\sC$ gives an exact sequence $Fy' \rightlabel{Ff} Fy \too Fy'' \too 0$ in $\mod \sS$ since $\Sigma y' \in \Sigma \sC = \Sigma \sS * \Sigma^2 \sS$. We thus obtain a short exact sequence $0 \to \im Ff \to Fy \to Fy'' \to 0$. Since $\cT$ is a torsion class, it follows that $\im Ff \in \cT$ and $Fy \in \cT$. Hence, we obtain $y \in \cY$.

For $s \in \sS$ we will construct a triangle $s \to y \to x \to \Sigma s$ in which $y \in \cY$ and $x \in \cX$, which will allow us to apply Lemma~\ref{lem:complete} to conclude that $(\cX,\cY)$ is a complete cotorsion pair in $\sC$.
First, we need to check that Lemma~\ref{lem:complete} applies. Since $\cY \subseteq \sC$, we have that $\sT(\sS,\Sigma \cY) = 0$ so that $\sS \subseteq \cX$. Furthermore, if $s \in \sS$ then $F(\Sigma s) = 0 \in \cT$ so that $\Sigma s \in \cY$. By Remark~\ref{rem:complete}\eqref{rem:complete-1}, Lemma~\ref{lem:complete} applies.

Now let $s \in \sS$ and consider $Fs$ and take a left $\cT$-approximation, $\phi \colon Fs \to t$, in $\mod \sS$. Let $y$ and $f \colon s \to y$ be such that $Fy \simeq t$ and $Ff = \phi$. We claim that $f \colon s \to y$ is a left $\cY$-approximation of $s$. Consider a morphism $g \colon s \to y'$ with $y' \in \cY$. Applying $F$ to $f$ and $g$ gives a diagram,
\[
\xymatrix{
Fs \ar[rr]^-{Ff = \phi} \ar[drr]_-{Fg = \theta} & & Fy \simeq t \, ,\ar@{-->}[d]^-{\psi} \\
                                                                    & & Fy' 
}
\]
where the morphism $\psi \colon t \to Fy'$ exists because $\phi$ is a left $\cT$-approximation.
Since the functor $F$ is full, there exists $h \colon y \to y'$ such that $Fh = \psi$.
We claim that $g = hf$. Applying $F$ to $g - hf$ shows that $g - hf = 0$ in $\sC/\Sigma \sS$. Hence, $g-hf$ factors through $\Sigma \sS$.
\[
\xymatrix@!R=4px{
s \ar[rr]^-{g - hf} \ar[dr] &                            & y' \\
                                     & \Sigma s' \ar[ur] &
}
\]
Hence, $g - hf = 0$ in $\sC$ since $\sS$ is presilting. 
It follows that $f \colon s \to y$ is a left $\cY$-approximation of $\sS$. Without loss of generality, we may assume that $f \colon s \to y$ is a minimal left $\cY$-approximation and extend it to a distinguished triangle, $s \rightlabel{f} y \too x \too \Sigma s$.
By the Wakamatsu lemma for triangulated categories, Lemma~\ref{lem:wakamatsu}, we see that $x \in {}\orth (\Sigma \cY) = \cX$.
\end{proof}

\begin{theorem} \label{thm:cotorsion-torsion}
Suppose the hypotheses of Setup~\ref{setup} hold.
Then, there is a bijection
\[
\{\text{complete cotorsion pairs in } \sC \} \bij
\{\text{functorially finite torsion pairs in } \mod \sS \}.
\]
\end{theorem}

\begin{proof}
We show that the maps $\Phi$ and $\Theta$ defined in Propositions~\ref{prop:cotorsion-to-torsion} and \ref{prop:torsion-to-cotorsion} are mutually inverse. Let $(\cX,\cY)$ be a complete cotorsion pair in $\sC$. 
By Proposition~\ref{prop:cotorsion-to-torsion} we have $\Phi \big( (\cX,\cY) \big) = (\cT,\cF)$, where $\cT = F\cY$ and $\cF = \cT\orth$ in $\mod \sS$. 
Applying $\Theta$ to $(\cT,\cF)$ produces a complete cotorsion pair $(\cX',\cY')$ in which $\cY' = \{ c \in \sC \mid Fc \in \cT\}$.
Clearly, $\cY \subseteq \cY'$.
To see that $\cY' \subseteq \cY$ take $y' \in \cY'$ and observe that there is an isomorphism $\phi \colon Fy \to Fy'$ in $\mod \sS$ for some $y\in \cY$ by the definition of $\cT$. 
Now, applying the same argument used to show that $F\cY$ is closed under quotients in the proof of Proposition~\ref{prop:cotorsion-to-torsion}, shows that $\sT(\cX,\Sigma y') = 0$, whence by completeness of the cotorsion pair $(\cX,\cY)$ we get $y' \in \cY$.
The equality $\Phi \Theta = 1$ follows from $F(F^{-1}(\cT))=\cT$.
\end{proof}


\end{document}